\documentclass[a4,12pt]{amsart}
\usepackage{amssymb,amstext,amsmath,amscd,amsthm,amsfonts,enumerate,latexsym}
\usepackage{color}
\usepackage[dvipdfmx]{graphicx}
\usepackage[all]{xy}

\theoremstyle{plain}
\newtheorem{thm}{Theorem}[section]

\newtheorem*{thm*}{Theorem}
\newtheorem*{cor*}{Corollary}

\newtheorem{prop}[thm]{Proposition}
\newtheorem{proposition}[thm]{Proposition}

\newtheorem{lem}[thm]{Lemma}
\newtheorem{cor}[thm]{Corollary}

\newtheorem*{claim*}{Claim}
\newtheorem*{acknowledgement}{Acknowledgement}
\theoremstyle{definition}

\newtheorem{definition}[thm]{Definition}
\newtheorem{ex}[thm]{Example}

\theoremstyle{remark}

\numberwithin{equation}{thm}

\def\Im{\operatorname{Im}}

\def\Ker{\operatorname{Ker}}

\def\RHom{\mathrm{{\bf R}Hom}}

\def\mod{\mathrm{mod}}

\def\rank{\mathrm{rank}}
\def\a{\mathrm a}

\def\m{\mathfrak m}
\def\n{\mathfrak n}

\newcommand{\rmr}{\mathrm{r}}

\newcommand{\rmK}{\mathrm{K}}

\newcommand{\calX}{\mathcal{X}}

\newcommand{\mapright}[1]{%
\smash{\mathop{%
\hbox to 1cm{\rightarrowfill}}\limits^{#1}}}

\newcommand{\mapleft}[1]{%
\smash{\mathop{%
\hbox to 1cm{\leftarrowfill}}\limits_{#1}}}

\title[The structure of Ulrich ideals in hypersurfaces]{The structure of Ulrich ideals in hypersurfaces}

\author{Ryotaro Isobe}
\thanks{Department of Mathematics and Informatics, Graduate School of Science and Engineering, Chiba University, Yayoi-cho 1-33, Inage-ku, Chiba, 263-8522, Japan}
\thanks{{\em E-mail address.} r.isobe.math@gmail.com}

\thanks{2010 {\em Mathematics Subject Classification.} 13D02, 13H10, 13H15.}
\thanks{{\em Key words and phrases.} Cohen-Macaulay ring, hypersurface ring, Ulrich ideal, Ulrich module, minimal free resolution, matrix factorization}

\begin{document}
\maketitle

\setlength{\baselineskip}{15pt}

\begin{abstract}
This paper studies Ulrich ideals in hypersurface rings. A characterization of Ulrich ideals is given. Using this characterization, we construct a minimal free resolution of an Ulrich ideal concretely. We also explore Ulrich ideals in a hypersurface ring of the form $R = k[[X, Y]]/(f)$.   
\end{abstract}



\section{Introduction}
The purpose of this paper is to investigate the structure and ubiquity of Ulrich ideals in a hypersurface ring.  

In a Cohen-Macaulay local ring $(R, \m)$, an $\m$-primary ideal $I$ is called an Ulrich ideal in $R$ if there exists a parameter ideal $Q$ of $R$ such that $I \supsetneq Q$, $I^2 = QI$, and $I/I^2$ is $R/I$-free. The notion of Ulrich ideal/module dates back to the work \cite{GOTWY1} in 2014, where S. Goto, K. Ozeki, R. Takahashi, K.-i. Watanabe, and K.-i. Yoshida introduced the notion, generalizing that of maximally generated maximal Cohen-Macaulay modules (\cite{BHU}), and started the basic theory. The maximal ideal of a Cohen-Macaulay local ring with minimal multiplicity is a typical example of Ulrich ideals, and the higher syzygy modules of Ulrich ideals are Ulrich modules. In \cite{GOTWY1, GOTWY2}, all Ulrich ideals of Gorenstein local rings of finite CM-representation type with dimension of at most 2 are determined by means of the classification in the representation theory. In \cite{GTT2}, S. Goto, R. Takahashi, and N. Taniguchi studied the structure of the complex $\RHom_R(R/I, R)$ for Ulrich ideals $I$ in a Cohen-Macaulay local ring of arbitrary dimension, and proved that in a one-dimensional non-Gorenstein almost Gorenstein local ring ($R, \m$), the only possible Ulrich ideal is the maximal ideal $\m$ (\cite[Theorem 2.14]{GTT2}). In contrast, in \cite{GIK}, S. Goto, the author, and S. Kumashiro closely explored the structure of chains of Ulrich ideals in a one-dimensional Cohen-Macaulay local ring. They studied the structure of the set $\calX_R$ of Ulrich ideals in $R$, and explored the ubiquity of those when $R$ is a generalized Gorenstein ring and $R$ has minimal multiplicity. Recently, S. Goto, the author, and N. Taniguchi \cite{GIT} explored Ulrich ideals in a one-dimensional 2-AGL ring.      

However, even for the case of hypersurface rings, there seems to be only scattered results known which give a complete list of Ulrich ideals, except the case of finite CM-representation type and the case of several numerical semigroup rings. Therefore, in the current paper, we focus our attention on a hypersurface ring which is not necessarily finite CM-representation type.

The main result of this paper is to give a characterization of Ulrich ideals for hypersurface rings of positive dimension. Let $(S, \n)$ be a regular local ring with $\dim S = d+1$ $($$d\ge1$$)$, and $f \in \n$ a non-zero element in $S$. We set $R = S/(f)$. For each $a\in S$, let $\overline{a}$ denote the image of $a$ in $R$. We denote by $\calX_R$ the set of Ulrich ideals in $R$. In a hypersurface ring $R$, every Ulrich ideal can be represented as an image of parameter ideals of $S$, as shown below.  

\begin{thm}$(${\rm Theorem \ref{2.2}}$)$
Suppose that $(S, \n)$ is a regular local ring with $\dim S = d+1$ $($$d\ge1$$)$ and $0 \neq f \in \n$. Set $R = S/(f)$. Then we have
$$\calX_R =
\left \{
(\overline{a_1}, \overline{a_2}, \cdots, \overline{a_d}, \overline{b}) \middle|\  
\parbox{27em}{\text{$a_1, a_2, \ldots, a_d, b\in \n$ be a system of parameters of $S$,}\\
\text{and there exist $x_1, x_2, \ldots, x_d \in (a_1, a_2, \cdots, a_d, b)$ and $\varepsilon \in U(S)$}\\
\text{such that $b^2 + \displaystyle\sum_{i=1}^d a_ix_i = \varepsilon f$.}}
\right \},
$$
where $U(S)$ denotes the set of unit elements of $S$.
\end{thm}

This theorem enables us to make a complete list of Ulrich ideals in a hypersurface ring $R$ which is not necessarily finite CM-representation type.  

We now explain how this paper is organized. In Section 2, we will summarize a few results and basic properties of Ulrich ideals, which we subsequently need. The proof of Theorem \ref{2.2} will be given in Section 3. In Section 4, we construct a  minimal free resolution of Ulrich ideals in a hypersurface ring $R = S/(f)$. Because every Ulrich ideal $I$ is an image of some parameter ideal in a regular local ring $S$, the resolution of $R/I$ can be constructed by using Tate's construction (\cite[Theorem 4]{T}). However, in this section, we give another construction based on \cite{GOTWY1, GTT2} in which the structure of minimal free resolutions of Ulrich ideals was closely explored. We also give a matrix factorization of the $d$-th syzygy module of $R/I$, which is an Ulrich module with respect to $I$ (Corollary \ref{3.5}). 
 In Section 5, we consider the structure of decomposable Ulrich ideals. We characterize decomposable 2-generated Ulrich ideals in a one-dimensional Cohen-Macaulay local ring, and determine all of these in a hypersurface ring.   
 In the last section, we focus our attention on the case of $S = k[[X, Y]]$ which is the formal power series ring over a field $k$. The purpose of this section is to make a complete list of Ulrich ideals in $R$ which is not finite CM-representation type. We give the list for the case of $f = Y^k$ and  $f = X^kY$ (Proposition \ref{5.1}, Theorem \ref{4.6}, Corollary \ref{4.8}, Theorem \ref{4.10}, Theorem \ref{4.13}, and Theorem \ref{4.17}).

Throughout this paper, let $\rmr(R)$ denote the Cohen-Macaulay type of $R$, and $\mu_R(M)$ (resp. $\ell_R(M)$) denote the number of elements in a minimal system of generators of $M$ (resp. the length of $M$), for a finitely generated $R$-module $M$. We denote by $\calX_R$ the set of Ulrich ideals in $R$.

\section{Basic facts}
Let us summarize a few results and basic properties of Ulrich ideals. We begin with the definition of Ulrich ideals. Although we focus our attention on the case of hypersurface rings, we would like to state the definition in the case of arbitrary Cohen-Macaulay local rings. Let $(R, \m)$ be a Cohen-Macaulay local ring with $\dim R = d \ge0$, and $I$ an $\m$-primary ideal of $R$. We assume that $I$ contains a parameter ideal $Q$ of $R$ as a reduction.
\begin{definition}(\cite[Definition 1.1]{GOTWY1})\label{1.1} 
We say that $I$ is an {\it Ulrich ideal} in $R$, if the following conditions are satisfied. 
\begin{enumerate}[$(1)$]
\item
$I \neq Q$, but $I^2 = QI$.
\item
$I/I^2$ is a free $R/I$-module.
\end{enumerate}
\end{definition}

In Definition \ref{1.1}, Condition $(1)$ is equivalent to saying that the associated graded ring $gr_I(R) = \oplus_{n\ge 0} I^n/I^{n+1}$ is a Cohen-Macaulay ring with $a(gr_I(R)) = 1-d$, where $a(gr_I(R))$ denotes the a-invariant of $gr_I(R)$ (\cite[Remark 3.10]{GS}, \cite[Remark 3.1.6]{GW}). Therefore, Condition $(1)$ is independent of the choice of reductions $Q$ of $I$. In addition, Condition $(2)$ is equivalent to saying that $I/Q$ is a free $R/I$-module, provided Condition $(1)$ is satisfied (\cite[Lemma 2.3]{GOTWY1}). If $I = \m$, then Condition $(2)$ is automatically satisfied. Hence, when the residue class field $R/\m$ of $R$ is infinite, the maximal ideal $\m$ is an Ulrich ideal if and only if $R$ is not a regular local ring, possessing minimal multiplicity (\cite{S}).     


For a finitely generated $R$-module $M$, we denote by  G-$\dim_RM$ the Gorenstein dimension (G-dimension for short) of $M$. With this notation, we then have the following.

\begin{thm}$($\cite[Theorem 7.1, Theorem 7.6]{GOTWY1}, \cite[Theorem 2.5, Theorem 2.8]{GTT2}$)$ \label{1.2}
Let $I$ be an Ulrich ideal in a Cohen-Macaulay local ring $R$, and set $n = \mu_R(I)$. Let
$$
\cdots \to F_i \overset{\partial_i}{\to} F_{i-1} \to \cdots \to F_1 \overset{\partial_1}{\to} F_0 = R \to R/I \to 0
$$
be a minimal free resolution of $R/I$. Then, setting $t = n-d$, the following assertions hold true.    
\begin{enumerate}[$(1)$]
\item
$t\cdot \rmr(R/I) = \rmr(R)$.
\item
$\mathbf{I}(\partial_i) = I$ for all $i \ge 1$.
\item
For $i \ge 0$, $\beta_i = 
\begin{cases}
t^{i-d}{\cdot}(t+1)^d & \ \ (i \ge d),\\
\binom{d}{i}+ t{\cdot}\beta_{i-1} & \ \ (1 \le i \le d),\\
1  & \ \  (i=0).
\end{cases}$
\item
$n = d+1$ if and only if G-$\dim_R R/I < \infty$.
\end{enumerate}
Here, $\mathbf{I}(\partial_i)$ denotes the ideal of $R$ generated by the entries of the matrix $\partial_i$, and $\beta_i = \rank_RF_i$. 
\end{thm}

Therefore, when $R$ is a Gorenstein ring, every Ulrich ideal $I$ is generated by $d+1$ elements, if it exists,  and $R/I$ has finite G-dimension but infinite projective dimension. 
Moreover, because $I/Q$ is a free $R/I$-module, we have $I = Q:_RI$, that is $I$ is a {\it good ideal} in the sense of \cite{GIW}. Similar to good ideals, Ulrich ideals are characteristic ideals, but behave very well in their nature (\cite{GOTWY1, GOTWY2}).
\section{Ulrich ideals in hypersurfaces}

In this section, we give a characterization of Ulrich ideals in a hypersurface ring. Firstly, let $(S, \n)$ be a Cohen-Macaulay local ring with $\dim S = d+1$ ($d\ge 1$), and $f\in \n$ a non-zero divisor on $S$. We set $R = S/(f)$ and $\m = \n/(f)$. For each $a\in S$, let $\overline{a}$ denote the image of $a$ in $R$, and $U(S)$ denote the set of unit elements of $S$. We then have the following.

\begin{prop}\label{2.1} 

Let $a_1, a_2, \ldots, a_d, b\in \n$ be a system of parameters of $S$. Suppose that there exist $x_1, x_2, \ldots, x_d \in (a_1, a_2, \cdots, a_d, b)$ and $\varepsilon \in U(S)$ such that 
$b^2 + \displaystyle\sum_{i=1}^d a_ix_i = \varepsilon f$. Then $I= (\overline{a_1}, \overline{a_2}, \cdots, \overline{a_d}, \overline{b}) \in \calX_R$.
 
\end{prop}
\begin{proof}
Since $a_1, \ldots, a_d, b$ is a system of parameters of $S$, $I$ is an $\m$-primary ideal of $R$. Let $Q = (\overline{a_1}, \cdots, \overline{a_d})$. Then $\overline{b}^2 \in QI$, since $b^2 + \displaystyle\sum_{i=1}^d a_ix_i = \varepsilon f$, therefore $I^2 = QI$. It suffices to show that $I/Q \cong R/I$ (see \cite[Lemma 2.3]{GOTWY1}). Since $I/Q$ is a homomorphic image of $R/I$, it is enough to show that $\ell_R(R/I) = \ell_R(I/Q)$, which is equivalent to $\ell_R(R/Q) = 2\cdot \ell_R(R/I)$. In fact, we have
$$\ell_R(R/Q) = \ell_S(S/(a_1, \cdots, a_d, f)) = \ell_S(S/(a_1, \cdots, a_d, b^2)) = 2\cdot \ell_R(R/I),
$$ 
where the second equality follows from the relation $b^2 + \displaystyle\sum_{i=1}^d a_ix_i = \varepsilon f$, and the third equality follows from the assumption that $a_1, \ldots, a_d, b$ is a system of parameters of $S$.
\end{proof}
The converse of Proposition \ref{2.1} is also true if $S$ is a regular local ring. The following is the main result of this section.

\begin{thm}\label{2.2}
Suppose that $(S, \n)$ is a regular local ring. Then we have
$$\calX_R =
\left \{
(\overline{a_1}, \overline{a_2}, \cdots, \overline{a_d}, \overline{b}) \middle|\  
\parbox{27em}{\text{$a_1, a_2, \ldots, a_d, b\in \n$ be a system of parameters of $S$,}\\
\text{and there exist $x_1, x_2, \ldots, x_d \in (a_1, a_2, \cdots, a_d, b)$ and $\varepsilon \in U(S)$}\\
\text{such that $b^2 + \displaystyle\sum_{i=1}^d a_ix_i = \varepsilon f$.}}
\right \}.
$$

\end{thm}

\vspace{1em}
In order to  prove Theorem \ref{2.2}, we need the following lemma learnt from Professor K.-i. Yoshida.

\begin{lem}\label{2.3}
Suppose that $S$ is a regular local ring. Assume that $a_1, a_2, \ldots, a_d, b \in \n$ and $(\overline{a_1}, \overline{a_2}, \cdots, \overline{a_d}, \overline{b}) \in \calX_R$.
Then $f\in (a_1, a_2, \cdots, a_d, b)^2$, and therefore $a_1, a_2, \ldots, a_d, b$ is a system of parameters of $S$.
\end{lem}
\begin{proof}
Set $I = (\overline{a_1}, \overline{a_2}, \cdots, \overline{a_d}, \overline{b})$. We look at the minimal free resolution 
$$
{F} : \cdots \to {F}_i \overset{\partial_i}{\to} {F}_{i-1} \to \cdots \to {F}_1 \overset{\partial_1}{\to} {F}_0 = R \overset{\varepsilon}{\to} R/I \to 0
$$
of $R/I$ and set $M = \Im \partial_d$. Since $R = S/(f)$ is a hypersurface ring, there exist matrices $A, B \in M_n(S)$ such that 
$
0 \to S^{\oplus n} \overset{A}{\to} S^{\oplus n} \overset{\varepsilon}{\to} M \to 0\ \text{is exact as $S$-modules and}\ AB = BA = fE_n,
$ 
 where n = $\mu_R(M)$ and $E_n \in M_n(S)$ is a unit matrix. Whence 
 $\cdots \to R^{n} \overset{\overline{B}}{\to} R^n \overset{\overline{A}}{\to} R^n \overset{\overline{B}}{\to} R^n \overset{\overline{A}}{\to} R^n \overset{\varepsilon}{\to} M \to 0$
 is a minimal free resolution of $M$. Therefore, we have $\mathbf{I}(\overline{A}) = \mathbf{I}(\overline{B}) = I$ in $R$ by \cite[Theorem 7.6]{GOTWY1}, that is $\mathbf{I}(A) \subseteq (a_1, \cdots, a_d, b) + (f)$ and $\mathbf{I}(B) \subseteq (a_1, \cdots, a_d, b) + (f)$ in $S$, where $\mathbf{I}(*)$ denotes the ideal of $R$ generated by the entries of the matrix $*$. Since $AB = fE_n$, we get 
 $$
 f\in \mathbf{I}(A) \cdot \mathbf{I}(B) \subseteq [(a_1, \cdots, a_d, b) + (f)]^2 = (a_1, \cdots, a_d, b)^2 + f[(a_1, \cdots, a_d, b) + (f)],
 $$
thus $f\in (a_1, \cdots, a_d, b)^2$ by Nakayama's lemma.
\end{proof}

We are now ready to prove Theorem \ref{2.2}.

\begin{proof}[Proof of Theorem \ref{2.2}]
Thanks to Proposition \ref{2.1}, we have only to show the inclusion $($$\subseteq$$)$. Let $I \in \calX_R$. Since $\mu_R(I) = d+1$ by Theorem \ref{1.2} (1), we can choose $a_1, \ldots, a_d, b \in \n$ so that $I = (\overline{a_1}, \cdots, \overline{a_d}, \overline{b})$, and $I^2 = (\overline{a_1}, \cdots, \overline{a_d})I$. Then, by Lemma \ref{2.3}, $a_1, \ldots, a_d, b$ is a system of parameters of $S$ and $f\in (a_1, a_2, \cdots, a_d, b)^2$. We write $f = \sum_{i=1}^da_iy_i + \delta b^2$ with $y_1, \ldots, y_d \in (a_1, \cdots, a_d, b)$ and $\delta\in S$. We then have 
\begin{equation*}
\begin{split}
\ell_R(R/Q) & = \ell_S(S/(a_1,\cdots, a_d, f)) = \ell_S(S/(a_1,\cdots, a_d, \delta b^2))\\
& = \ell_S(S/(a_1, \cdots, a_d, \delta)) + 2\cdot \ell_R(R/I).
\end{split}
\end{equation*}
Because $I \in\calX_R$ and $\mu_R(I) = d+1$, we have $I/Q \cong R/I$, whence $\ell_R(R/Q) = 2\cdot \ell_R(R/I)$ (see the proof of Proposition \ref{2.1}). Therefore, $\ell_S(S/(a_1, \cdots, a_d, \delta)) = 0$, that is $\delta \in U(S)$. Setting $x_i = \delta^{-1} y_i$ $($$\in (a_1, \cdots, a_d, b)$$)$ and $\varepsilon = \delta^{-1}$, we get $b^2 + \displaystyle\sum_{i=1}^d a_ix_i = \varepsilon f$, which completes the proof of Theorem \ref{2.2}.
\end{proof}

The following is a direct consequence of Theorem \ref{2.2}, which gives many examples of Ulrich ideals.

\begin{cor}\label{2.4}
Suppose that $f = b^2$ for some $b\in \n$. Then, for any system of parameters $a_1, a_2, \ldots, a_d$ of $S/(b)$, we have $(\overline{a_1}, \overline{a_2}, \cdots, \overline{a_d}, \overline{b}) \in \calX_R$.
\end{cor}
\begin{proof}
We can put $x_i = 0$ and $\varepsilon = 1$.
\end{proof}

We will use Proposition \ref{2.1}, Theorem \ref{2.2}, and Corollary \ref{2.4} later in Section 6.

\section{Minimal free resolutions}
In this section, we construct a minimal free resolution of an Ulrich ideal $I$ which is obtained in Section 3. Because $I$ is an image of some parameter ideal, it is well known that this resolution can be constructed by using Tate's construction (\cite[Theorem 4]{T}). However, in this section, let us give another construction by using properties of Ulrich ideals. We begin with the following lemma.
\begin{lem}\label{3.1}
Suppose that $S$ is a commutative ring and $a_1,\ldots, a_d, x_1, \ldots, x_d \in S$ $($$d\ge1$$)$. We set
$K = {\rmK}_{\bullet}(a_1, \ldots, a_d;S) = (K_{\bullet}, \partial^{K}_{\bullet}) \ \text{and}\  L = {\rmK}_{\bullet}(x_1, \ldots, x_d;S) = (K_{\bullet}, \partial^{L}_{\bullet}) 
$
are Koszul complexes of $S$ generated by $a_1, \ldots, a_d$ and $x_1, \ldots, x_d$, and $c = \displaystyle\sum_{i = 1}^d a_ix_i$. Then
$$\partial^K_p \cdot {}^t\!{\partial^L_p} + {}^t\!{\partial^L_{p-1}} \cdot \partial^K_{p-1} = c\cdot id_{K_{p-1}}\ \text{for any}\ p \in \mathbb{Z},$$
where ${}^t\!*$ denotes the transpose of the matrix $*$.
\end{lem}
\begin{proof}
We may assume that $1\le p \le d+1$. If $p = 1$, 
$$
\partial_1^{K} = \left[
\begin{matrix}
a_1 & a_2 & \cdots & a_d
\end{matrix} \right],
{}^t\!{\partial_1^L} = \left[
\begin{matrix}
x_1\\
x_2\\
\vdots\\
x_d
\end{matrix}\right],
{}^t\!{\partial_0^L} = 0,
\text{and}\
\partial_0^K = 0,
$$
hence $\partial^K_1 \cdot {}^t\!{\partial^L_1} + {}^t\!{\partial^L_{0}} \cdot \partial^K_{0} = \partial^K_1 \cdot {}^t\!{\partial^L_1} = c$.
If $p = d+1$, 
$$
\partial_{d+1}^{K} = 0,
{}^t\!{\partial_{d+1}^L} = 0,
{}^t\!{\partial_d^L} = \left[
\begin{matrix}
x_1 & \cdots & (-1)^{i+1}x_i & \cdots & (-1)^{d+1}x_d
\end{matrix} \right],
\text{and}\
\partial_d^K = \left[
\begin{matrix}
a_1\\
\vdots\\
(-1)^{i+1}a_i\\
\vdots\\
(-1)^{d+1}a_d
\end{matrix}\right],
$$
hence $\partial^K_{d+1} \cdot {}^t\!{\partial^L_{d+1}} + {}^t\!{\partial^L_{d}} \cdot \partial^K_{d} = {}^t\!{\partial^L_d} \cdot \partial^K_d = c$.\\
We now assume that $2\le p \le d$. Set $K_1 = \sum_{i = 1}^d RT_i $, $\Lambda = \{ 1, 2, \cdots, d\}$, and $F_i = \{ I \subseteq \Lambda \mid \sharp I = i \}$ for $0 \le i \le d$. For $I = \{ j_1 < j_2 < \cdots < j_p\} \in F_p$, we denote $ T_{I} = T_{j_1} \wedge T_{j_2} \wedge \cdots \wedge T_{j_p}$. Then $K_p = \oplus_{I \in F_p} RT_I $, and the matrix $\partial_p^K$$($resp.  $\partial_p^L$$)$ has the following form
$$
[\partial_p^K]_{I, J}\ (\text{resp.}\ [\partial_p^L]_{I, J})=
\begin{cases}
0\ \ \ \text{if}\ I \nsubseteq J, \\
(-1)^{\alpha +1}a_{j_\alpha}\ (\text{resp.}\ (-1)^{\alpha +1}x_{j_\alpha})\ \ \ \text{if}\ \ 
\begin{split}
& I \subseteq J,\ J = \{ j_1< \cdots < j_p\},\\
& \text{and}\ I = J\setminus \{j_\alpha\},
\end{split}
\end{cases}
$$
for $I \in F_{p-1}$ and $J \in F_p$. Let us check the following. 
\begin{claim*}
For $I_1, I_2 \in F_{p-1}$, the following assertions hold true.
\begin{enumerate}[$(1)$]
\item
$\sharp (I_1\cup I_2) \ge p+1$ if and only if $\sharp (I_1 \cap I_2) \le p-3$.
\item
$\sharp (I_1\cup I_2) = p$ if and only if $\sharp (I_1 \cap I_2) = p-2$.
\item
$\sharp (I_1\cup I_2) \le p-1$ if and only if $\sharp (I_1 \cap I_2) \ge p-1$. When this is the case, $I_1 = I_2$.
\end{enumerate} 
\end{claim*}
\begin{proof}[Proof of Claim]
Focus on the number $\sharp (I_1\setminus I_2)$. $(1)$ is the case $\sharp (I_1\setminus I_2)\ge 2$, $(2)$ is $\sharp (I_1\setminus I_2) = 1$, otherwise $(3)$.
\end{proof}
It suffices to show that 
$$
[\partial^K_p \cdot {}^t\!{\partial^L_p} + {}^t\!{\partial^L_{p-1}} \cdot \partial^K_{p-1}]_{I_1, I_2} =
\begin{cases}
0\ \ \ \text{if  $\sharp(I_1 \cup I_2) \ge p+1$}\\
0\ \ \ \text{if  $\sharp(I_1 \cup I_2) = p $}\\
c\ \ \ \text{if  $\sharp(I_1 \cup I_2) \le p-1$}\\
\end{cases} 
$$
for any $I_1, I_2 \in F_{p-1}$ by Claim. We notice that 

\begin{equation*}
\begin{split}
[\partial^K_p \cdot {}^t\!{\partial^L_p} + {}^t\!{\partial^L_{p-1}} \cdot \partial^K_{p-1}]_{I_1, I_2} & = [\partial^K_p \cdot {}^t\!{\partial^L_p}]_{I_1, I_2} +[ {}^t\!{\partial^L_{p-1}} \cdot \partial^K_{p-1}]_{I_1, I_2}\\
& = \sum_{J \in F_p} [\partial_p^K]_{I_1, J} \cdot [\partial_p^L]_{I_2, J} + \sum_{J' \in F_{p-2}} [\partial_{p-1}^L]_{J', I_1} \cdot [\partial_{p-1}^K]_{J', I_2}\\
& = \sum_{ J \in F_p,\ I_1 \cup I_2\subseteq J} [\partial_p^K]_{I_1, J} \cdot [\partial_p^L]_{I_2, J} + \sum_{J' \in F_{p-2},\ J' \subseteq I_1\cap I_2} [\partial_{p-1}^L]_{J', I_1} \cdot [\partial_{p-1}^K]_{J', I_2}.
\end{split}
\end{equation*}

If  $\sharp(I_1 \cup I_2) \ge p+1$, then $\{ J \in F_p \mid I_1 \cup I_2 \subseteq J\} = \emptyset$ and $\{ J' \in F_{p-2} \mid J' \subseteq I_1 \cap I_2\} = \emptyset$ by Claim. Therefore
$[\partial^K_p \cdot {}^t\!{\partial^L_p} + {}^t\!{\partial^L_{p-1}} \cdot \partial^K_{p-1}]_{I_1, I_2}$ = 0.

If $\sharp(I_1 \cup I_2) = p$, we set $I_1 = \{ j_1< j_2< \cdots < j_{p-1}\}$ and $I_2 = \{ \ell_1< \ell_2< \cdots < \ell_{p-1}\}$, and take $j_\alpha \in I_1 \setminus I_2$ and $\ell_\beta \in I_2\setminus I_1$   
$($$1\le \alpha, \beta \le p-1$$)$. We then have
$$
\{ J \in F_p \mid I_1 \cup I_2 \subseteq J\} = \{ I_1\cup I_2\} = \{ I_1 \cup \{ \ell_\beta\}\} = \{ I_2 \cup \{ j_\alpha\}\},\ \text{and}
$$
$$
\{ J' \in F_{p-2} \mid J' \subseteq I_1 \cap I_2\} = \{ I_1\cap I_2\} = \{ I_1\setminus \{ j_\alpha\}\} = \{ I_2 \setminus \{ \ell_\beta\}\},
$$
hence we get 
\begin{equation*}
\begin{split}
[\partial^K_p \cdot {}^t\!{\partial^L_p}]_{I_1, I_2} & =
[\partial_p^K]_{I_1, I_1\cup \{ \ell_\beta\}} \cdot [\partial_p^L]_{I_2, I_2 \cup \{ j_\alpha\}}\\
& = 
\begin{cases}
(-1)^{\beta +1}a_{\ell_\beta}\cdot (-1)^{\alpha + 2}x_{j_\alpha} \ \ \ \text{if $j_\alpha > \ell_\beta$}\\
(-1)^{\beta +2}a_{\ell_\beta}\cdot (-1)^{\alpha + 1}x_{j_\alpha} \ \ \ \text{if $j_\alpha < \ell_\beta$} 
\end{cases}\\
& = (-1)^{\alpha + \beta + 1}a_{\ell_\beta}x_{j_\alpha},\ \text{and}
\end{split}
\end{equation*}
\begin{equation*}
\begin{split}
[ {}^t\!{\partial^L_{p-1}} \cdot \partial^K_{p-1}]_{I_1, I_2}& =
[\partial_{p-1}^L]_{I_1 \setminus \{ j_\alpha\}, I_1} \cdot [\partial_{p-1}^K]_{I_2 \setminus \{ \ell_\beta\}, I_2}\\
& = (-1)^{\alpha + 1} x_{j_\alpha} \cdot (-1)^{\beta +1} a_{\ell_\beta} \\
& = (-1)^{\alpha + \beta} a_{\ell_\beta} x_{j_\alpha}.
\end{split}
\end{equation*}
Therefore $[\partial^K_p \cdot {}^t\!{\partial^L_p} + {}^t\!{\partial^L_{p-1}} \cdot \partial^K_{p-1}]_{I_1, I_2}$ = 0. 

If $\sharp(I_1 \cup I_2) \le p-1$, then $I_1 = I_2$, whence
$$
\{ J \in F_p \mid I_1 \cup I_2 \subseteq J\} = \{ I_1 \cup \{ j\} \mid j \in \Lambda \setminus I_1\},\ \text{and}
$$
$$
\{ J' \in F_{p-2} \mid J' \subseteq I_1 \cap I_2\} = \{ I_1\setminus \{ j\} \mid j \in I_1\}.
$$
Hence we get
\begin{equation*}
\begin{split}
[\partial^K_p \cdot {}^t\!{\partial^L_p}]_{I_1, I_2} & = \sum_{j \in \Lambda \setminus I_1}[\partial_p^K]_{I_1, I_1 \cup \{ j\}} \cdot [\partial_p^L]_{I_1, I_1 \cup \{ j\}}\\
& = \sum_{j \in \Lambda \setminus I_1}a_jx_j, \ \text{and}
\end{split}
\end{equation*}
\begin{equation*}
\begin{split}
[ {}^t\!{\partial^L_{p-1}} \cdot \partial^K_{p-1}]_{I_1, I_2}& = \sum_{j \in I_1} [\partial_{p-1}^L]_{I_1 \setminus \{ j\}, I_1} \cdot [\partial_{p-1}^K]_{I_1 \setminus \{ j\}, I_1}\\
& = \sum_{j \in I_1} a_jx_j.
\end{split}
\end{equation*}
We then have
$[\partial^K_p \cdot {}^t\!{\partial^L_p} + {}^t\!{\partial^L_{p-1}} \cdot \partial^K_{p-1}]_{I_1, I_2} = \sum_{j \in \Lambda \setminus I_1}a_jx_j + \sum_{j \in I_1} a_jx_j = c$.
\end{proof}

In what follows, throughout this section, we assume that $(S, \n)$ is a Cohen-Macaulay local ring with $\dim S = d+1$ $($$d\ge1$$)$, and $f\in \n$ a non-zero divisor on $S$. We set $R=S/(f)$.
Let $a_1, \ldots, a_d, b \in \n$ be a system of parameters of $S$ so that $b^2 + \sum_{i = 1}^d a_ix_i = \varepsilon f$ with $x_1, \ldots, x_d \in (a_1, \cdots, a_d, b)$ and $\varepsilon \in U(S)$. Then $I = (\overline{a_1}, \overline{a_2}, \cdots, \overline{a_d}, \overline{b}) \in \calX_R$ with a reduction $Q = (\overline{a_1}, \overline{a_2}, \cdots, \overline{a_d})$ by Proposition \ref{2.1}. We notice that every Ulrich ideal in $R$ is this form, if $S$ is a regular local ring (Theorem \ref{2.2}). We also notice that $I/Q \cong R/I$.
By \cite[Corollary 7.2]{GOTWY1}, in the exact sequence $0 \to Q \overset{\iota}{\to} I \to R/I \to 0 $, the free resolution of $I$ induced from minimal free resolutions of $Q$ and $R/I$ is also minimal. We construct this resolution by using the relation $b^2 + \sum_{i = 1}^d a_ix_i = \varepsilon f$.
We set
$$K = {\rmK}_{\bullet}(a_1, \ldots, a_d;S) = (K_{\bullet}, \partial^{K}_{\bullet}) \ \text{and}\  L = {\rmK}_{\bullet}(x_1, \ldots, x_d;S) = (K_{\bullet}, \partial^{L}_{\bullet}) 
$$
are Koszul complexes of $S$ generated by $a_1, \ldots, a_d$ and $x_1, \ldots, x_d$. We define the sequence  $G = (G_{\bullet}, \partial_{\bullet})$ by $G_0 = K_0$, $G_i = K_i\oplus G_{i-1}= S^{\oplus \sum_{j=0}^i{\binom{d}{j}}}$ for $i\ge 1$, and
\vspace{0.5em}
$$
\partial_1 =\left[ 
\begin{array}{c|c}
\partial^{K}_1 & b
\end{array}\right], 
\partial_2 =\left[ 
\begin{array}{c|c}
\partial^{K}_2 & -bE_d \mid  {}^t\!{\partial^{L}_1}\\ \hline
O & \partial_1
\end{array}\right],\ \text{and}
$$
$$
\partial_i =\left[ 
\begin{array}{c|c}
\partial^{K}_i & (-1)^{i-1}bE_{\binom{d}{i-1}} \mid  {}^t\!{\partial^{L}_{i-1}} \mid O\\ \hline
O & \partial_{i-1}
\end{array}\right]\ 
\text{for $i\ge3$}.
$$
Then $\partial_i = \partial_{d+1}$ for any $i\ge d+1$.
Set $F = ({F}_{\bullet}, \overline{\partial_{\bullet}}) = ({G}_{\bullet} \otimes R, \partial_{\bullet} \otimes R)$.
We then have the following, which is the main result of this section.

\begin{thm}\label{3.2}
${F} : \cdots \to {F}_i \overset{\overline{\partial_i}}{\to} {F}_{i-1} \to \cdots \to {F}_1 \overset{\overline{\partial_1}}{\to} {F}_0 = R \overset{\varepsilon}{\to} R/I \to 0$ is a minimal free resolution of $R/I$.
\end{thm}


To prove Theorem \ref{3.2}, we give the following proposition.

\begin{prop}\label{3.3}
Set $g = \varepsilon f$ $($$=b^2 + \sum_{i = 1}^d a_ix_i$$)$. Then
$$\partial_i \cdot \partial_{i+1} =\left[
\begin{array}{c|c}
O & gE_{\sum_{j=0}^{i-1}\binom{d}{j}}
\end{array}\right]\ 
\text{for any $i\ge 1$}.
$$
In particular, $\partial_{d+1}^2 = gE_{2^d}$.
\end{prop}
\begin{proof}
By direct computation and Lemma \ref{3.1}, we have
\begin{equation*}
\begin{split}
\partial_1 \cdot \partial_2 & = \left[
\begin{array}{c|c}
\partial_1^K & b
\end{array} 
\right] \cdot
\left[
\begin{array}{c|c}
\partial_2^K & -bE_d \mid {}^t\!{\partial_1^L}\\ \hline
 O & \partial_1
\end{array} 
\right]\\
& = \left[
\begin{array}{c|c}
\partial_1^K & b
\end{array} 
\right] \cdot
\left[
\begin{array}{c|c|c}
\partial_2^K & -bE_d & {}^t\!{\partial_1^L}\\ \hline
 O & \partial_1^K & b
\end{array} 
\right]\\
& = \left[
\begin{array}{c|c|c}
O & O & \partial_1^K\cdot {}^t\!{\partial_1^L} + b^2
\end{array} 
\right] = \left[
\begin{array}{c|c}
O & g
\end{array} 
\right],
\end{split}
\end{equation*}

\begin{equation*}
\begin{split}
\partial_2 \cdot \partial_3 & = \left[
\begin{array}{c|c|c}
\partial_2^K & -bE_d & {}^t\!{\partial_1^L}\\ \hline
 O & \partial_1^K & b
\end{array} 
\right] \cdot
\left[
\begin{array}{c|c|c|c}
\partial_3^K & bE_{\binom{d}{2}} & {}^t\!{\partial_2^L} & O\\ \hline
O & \partial_2^K & -bE_d & {}^t\!{\partial_1^L} \\ \hline
O & O & \partial_1^K & b
\end{array} 
\right]\\
& = \left[
\begin{array}{c|c|c|c}
O & O & \partial^K_2 \cdot {}^t\!{\partial^L_2} + {}^t\!{\partial^L_{1}} \cdot \partial^K_{1} + b^2E_d & O\\ \hline
O & O & O & \partial_1^K \cdot {}^t\!{\partial_1^L} + b^2
\end{array} 
\right] = \left[
\begin{array}{c|c}
O & gE_d
\end{array} 
\right],
\end{split}
\end{equation*}

\begin{equation*}
\begin{split}
\partial_3 \cdot \partial_4 & = \left[
\begin{array}{c|c|c|c}
\partial_3^K & bE_{\binom{d}{2}} & {}^t\!{\partial_2^L} & O\\ \hline
O & \partial_2^K & -bE_d & {}^t\!{\partial_1^L} \\ \hline
O & O & \partial_1^K & b
\end{array} 
\right] \cdot
\left[
\begin{array}{c|c|c|c|c}
\partial_4^K & -bE_{\binom{d}{3}} & {}^t\!{\partial_3^L} & O & O\\ \hline
O & \partial_3^K & bE_{\binom{d}{2}} & {}^t\!{\partial_2^L} & O\\ \hline
O & O & \partial_2^K & -bE_d & {}^t\!{\partial_1^L} \\ \hline
O & O & O & \partial_1^K & b
\end{array} 
\right]\\
& = \left[
\begin{array}{c|c|c|c|c}
O & O & \partial_3^K \cdot {}^t\!{\partial_3^L} + {}^t\!{\partial_2^L} \cdot \partial_2^K + b^2E_{\binom{d}{2}} & O & O\\ \hline
O & O & O & \partial^K_2 \cdot {}^t\!{\partial^L_2} + {}^t\!{\partial^L_{1}} \cdot \partial^K_{1} + b^2E_d & O\\ \hline
O & O & O & O & \partial_1^K \cdot {}^t\!{\partial_1^L} + b^2
\end{array} 
\right]\\ 
& = \left[
\begin{array}{c|c}
O & gE_{\sum_{j = 0}^2 \binom{d}{j}}
\end{array} 
\right].
\end{split}
\end{equation*}
 Hence, we may assume that  $i \ge 4$ and our assertion holds true for $i-1$. 
Let $A_j = \left[ (-1)^{j-1} bE_{\binom{d}{j-1}} \mid {}^t\!{\partial_{j-1}^L} \mid O \right]$ for $j\ge 1$. Then

\begin{equation*}
\begin{split}
\partial_i \cdot \partial_{i+1} & = \left[
\begin{array}{c|c}
\partial_i^K & A_i\\ \hline
O & \partial_{i-1}
\end{array} 
\right] \cdot
\left[
\begin{array}{c|c}
\partial_{i+1}^K & A_{i+1}\\ \hline
 O & \partial_i
\end{array} 
\right]\\
& = \left[
\begin{array}{c|c}
O & \partial_i^K \cdot A_{i+1} + A_i \cdot \partial_i\\ \hline
O & \partial_{i-1} \cdot \partial_i
\end{array} 
\right] = \left[
\begin{array}{c|c}
O & \partial_i^K \cdot A_{i+1} + A_i \cdot \partial_i\\ \hline
O & O \mid gE_{\sum_{j = 0}^{i-2}\binom{d}{j}}
\end{array} 
\right], \ \text{and}
\end{split}
\end{equation*}

\begin{equation*}
\begin{split}
\partial_i^K \cdot A_{i+1} + A_i \cdot \partial_i & = \left[ (-1)^i b \partial_i^K \mid \partial_i^K\cdot {}^t\!{\partial_i^L} \mid O \right]\\
& + \left[ (-1)^{i-1}bE_{\binom{d}{i-1}} \middle | {}^t\!{\partial_{i-1}^L} \middle | O \right] \cdot \left[
\begin{array}{c|c|c}
\partial_i^K & (-1)^{i-1}bE_{\binom{d}{i-1}} & {}^t\!{\partial_{i-1}^L} \mid O\\ \hline
O & \partial_{i-1}^K & (-1)^{i-2}bE_{\binom{d}{i-2}} \mid O\\ \hline
O & O & \partial_{i-2}
\end{array}
\right] \\
& = \left[ (-1)^ib\partial_i^K \mid \partial_i^K\cdot {}^t\!{\partial_i^L} \mid O \right] + \left[ (-1)^{i-1}b\partial_i^K \mid {}^t\!{\partial_{i-1}^L} \cdot \partial_{i-1}^K + b^2E_{\binom{d}{i-1}} \mid O \right]\\
& = \left[ O \middle | \partial_i^K\cdot {}^t\!{\partial_i^L} + {}^t\!{\partial_{i-1}^L} \cdot \partial_{i-1}^K + b^2E_{\binom{d}{i-1}} \middle | O \right] = \left[ O \middle | gE_{\binom{d}{i-1}} \middle | O \right],
\end{split}
\end{equation*}
by Lemma \ref{3.1}. Therefore $\partial_i \cdot \partial_{i+1} = \left[ O \middle | gE_{\sum_{j=0}^{i-1} \binom{d}{j}} \right]$, as desired. 
\end{proof}

We are now ready to prove Theorem \ref{3.2}.
\begin{proof}[Proof of Theorem \ref{3.2}]
Thanks to Proposition \ref{3.3}, $\overline{\partial_i} \cdot \overline{\partial_{i+1}} = 0$ for all $i\ge1$, hence $F$ is a complex. Let $Q = (\overline{a_1}, \cdots, \overline{a_d})$.
Then $\overline{K} = (\overline{K_{\bullet}}, \overline{\partial_{\bullet}^K}) = (K_{\bullet} \otimes R, \partial_{\bullet}^K \otimes R)$ is a minimal free resolution of $R/Q$, since $Q$ is a parameter ideal of $R$, and $\overline{K}$ is a subcomplex of $F$. In contrast, $0 \to Q \overset{\iota}{\to} I \to R/I \to 0$ is exact and the following diagrams 
$$
\xymatrix{
0 \ar[r] &  \overline{K_0} \ar[r]^{\iota}\ar[d] & F_1 \ar[r]\ar[d]^{\overline{\partial_1}} & F_0 \ar[r] \ar[d]^{\varepsilon} & 0 \\
0 \ar[r] & Q \ar[r]^{\iota} \ar[d] & I \ar[r] \ar[d]& R/I \ar[r] \ar[d]& 0\\
 & 0 & 0 & 0 &  
}
$$
 and
$$
\xymatrix{
0 \ar[r] &  \overline{K_{i+1}} \ar[r]^{\iota}\ar[d]^{\overline{\partial_{i+1}^K}} & F_{i+1} \ar[r]\ar[d]^{\overline{\partial_{i+1}}} & F_i \ar[r] \ar[d]^{\overline{\partial_{i}}} & 0 \\
0 \ar[r] & \overline{K_{i}} \ar[r]^{\iota} \ar[d]^{\overline{\partial_{i}^K}} & F_i \ar[r] \ar[d]^{\overline{\partial_{i}}}& F_{i-1} \ar[r] \ar[d]^{\overline{\partial_{i-1}}}& 0\\
0 \ar[r]& \overline{K_{i-1}} \ar[r]^{\iota} & F_{i-1} \ar[r] & F_{i-2} \ar[r] & 0 
}
$$
\vspace{1em}

\noindent are commutative for all $i \ge 2$. Therefore $F$ is exact, whence $F$ is a minimal free resolution of $R/I$, since every entry of $\partial_{\bullet}$ is not a unit. This completes the proof of Theorem \ref{3.2}.
\end{proof}
As a consequence, we get a matrix factorization of $d$-th syzygy module of $R/I$, which is an Ulrich module with respect to $I$ (see \cite[Definition 1.2]{GOTWY1}).

\begin{cor}\label{3.5}
Let $M = \Im \overline{\partial_d}$. Then 
$0 \to G_{d+1} \overset{\partial_{d+1}}{\to} G_{d} \overset{\tau}{\to} M \to 0$
is exact as $S$-modules and
$\partial_{d+1}^2 = gE_{2^d}$, where $\tau : G_{d} \overset{\varepsilon}{\to} F_{d} \overset{\overline{\partial_d}}{\to} M$.
Therefore $\partial_{d+1}$ gives a matrix factorization of $M$. 
\end{cor}
\begin{proof}
Set $n = 2^d$.
Because $\partial_{d+1}^2 = gE_{n}$ (Proposition \ref{3.3}) and $g$ is a non-zero divisor on $S$, the map $ G_{d+1} \overset{\partial_{d+1}}{\to} G_{d} $ is injective. 
 It is clear that $\tau \circ \partial_{d+1} = 0$. Suppose that 
 \begin{equation*}
 \left( 
 \begin{matrix}
 x_1\\
  x_2\\
  \vdots\\
  x_n
 \end{matrix}
  \right) \in \Ker \tau.\
 \text{Then, since}\ \overline{\partial_d}\cdot 
   \left( 
 \begin{matrix}
 \overline{x_1}\\
  \overline{x_2}\\
  \vdots\\
  \overline{x_n}
 \end{matrix}
  \right) = 0\ \text{in $R$},\ 
    \left( 
 \begin{matrix}
 \overline{x_1}\\
  \overline{x_2}\\
  \vdots\\
  \overline{x_n}
 \end{matrix}
  \right) = 
\overline{\partial_d+1}\cdot 
   \left( 
 \begin{matrix}
 \overline{y_1}\\
  \overline{y_2}\\
  \vdots\\
  \overline{y_n}
 \end{matrix}
  \right)\
 \text{for some $y_i \in S$} 
\end{equation*}
by Theorem \ref{3.2}. Therefore
   \begin{equation*}
   \left( 
 \begin{matrix}
 x_1\\
 x_2\\
  \vdots\\
 x_n
 \end{matrix}
  \right)  = 
\partial_{d+1}\cdot 
   \left( 
 \begin{matrix}
 y_1\\
 y_2\\
  \vdots\\
 y_n
 \end{matrix}
  \right) + g\cdot
    \left( 
 \begin{matrix}
 z_1\\
 z_2\\
  \vdots\\
 z_n
 \end{matrix}
  \right)
   = 
\partial_{d+1}\cdot 
   \left( 
 \begin{matrix}
 y_1\\
 y_2\\
  \vdots\\
 y_n
 \end{matrix}
  \right) + \partial_{d+1}^2\cdot
    \left( 
 \begin{matrix}
 z_1\\
 z_2\\
  \vdots\\
 z_n
 \end{matrix}
  \right)
   =
  \partial_{d+1}\cdot 
  \left[
   \left( 
 \begin{matrix}
 y_1\\
 y_2\\
  \vdots\\
 y_n
 \end{matrix}
  \right) + \partial_{d+1}\cdot
    \left( 
 \begin{matrix}
 z_1\\
 z_2\\
  \vdots\\
 z_n
 \end{matrix}
  \right)
  \right],
\end{equation*}
for some $z_i \in S$. Consequently, $0 \to G_{d+1} \overset{\partial_{d+1}}{\to} G_d \overset{\tau}{\to} M \to 0$ is exact.
\end{proof}
We close this section with examples.
\begin{ex}\label{3.6}
\begin{enumerate}[$(1)$]
\item
If $d=1$, then
$$\partial_1 =\left[
\begin{matrix}
a_1 & b
\end{matrix}\right],\
\text{and}\
\partial_2 =\left[
\begin{matrix}
-b & x_1\\
a_1 & b
\end{matrix}\right].
$$
\item
If $d=2$, then
$$\partial_1 =\left[
\begin{matrix}
a_1 & a_2 & b
\end{matrix}\right],\
\partial_2 =\left[
\begin{matrix}
-a_2 & -b & 0 & x_1\\
a_1 & 0 & -b & x_2\\
0 & a_1 & a_2 & b
\end{matrix}\right],\
\text{and}\
\partial_3 =\left[
\begin{matrix}
b & -x_2 & x_1 & 0\\
-a_2 & -b & 0 & x_1\\
a_1 & 0 & -b & x_2\\
0 & a_1 & a_2 & b
\end{matrix}\right].
$$
\item
If $d=3$, then
$$\partial_1 =\left[
\begin{matrix}
a_1 & a_2 & a_3 &b
\end{matrix}\right],\
\partial_2 =\left[
\begin{matrix}
-a_2 & -a_3 & 0 & -b & 0 & 0 & x_1\\
a_1 & 0 & -a_3 & 0 & -b & 0 & x_2\\
0 & a_1 & a_2 & 0 & 0 & -b & x_3\\
0 & 0 & 0 & a_1 & a_2 & a_3 & b
\end{matrix}\right],
$$ 
$$
\partial_3 =\left[
\begin{matrix}
a_3 & b & 0 & 0 & -x_2 & x_1 & 0 & 0\\
-a_2 & 0 & b & 0 & -x_3 & 0 & x_1 & 0\\
a_1 & 0 & 0 & b & 0 & x_3 & x_2 & 0\\
0 & -a_2 & -a_3 & 0 & -b & 0 & 0 & x_1\\
0 & a_1 & 0 & -a_3 & 0 & -b & 0 & x_2\\
0 & 0 & a_1 & a_2 & 0 & 0 & -b & x_3\\
0 & 0 & 0 & 0 & a_1 & a_2 & a_3 & b
\end{matrix}\right],
$$
$$
\text{and}\
\partial_4 =\left[
\begin{matrix}
-b & x_3 & -x_2 & x_1 & 0 & 0 & 0 & 0\\
a_3 & b & 0 & 0 & -x_2 & x_1 & 0 & 0\\
-a_2 & 0 & b & 0 & -x_3 & 0 & x_1 & 0\\
a_1 & 0 & 0 & b & 0 & x_3 & x_2 & 0\\
0 & -a_2 & -a_3 & 0 & -b & 0 & 0 & x_1\\
0 & a_1 & 0 & -a_3 & 0 & -b & 0 & x_2\\
0 & 0 & a_1 & a_2 & 0 & 0 & -b & x_3\\
0 & 0 & 0 & 0 & a_1 & a_2 & a_3 & b
\end{matrix}\right].
$$
\end{enumerate}
\end{ex}

\section{Decomposable Ulrich ideals}

In this section, we explore the structure of decomposable Ulrich ideals. We begin with the following, which characterizes two-generated decomposable Ulrich ideals  in a one-dimensional Cohen-Macaulay local ring $R$.

\begin{prop}\label{4.1}
Suppose that $(R, \m)$ is a Cohen-Macaulay local ring with $\dim R = 1$. Let $I$ be an $\m$-primary ideal of $R$, and assume that $\mu_R(I) = 2$. Then the following conditions are equivalent.  
\begin{enumerate}[$(1)$]
\item
$I\in \calX_R$ and $I$ is decomposable. 
\item
There exist $a, b \in \m$ such that $I = (a, b)$, $(a) = (0):_Rb$, and $(b) = (0):_Ra$. 
\end{enumerate}
\end{prop}
\begin{proof}
(1) $\Rightarrow$ (2)  Choose $a, b \in \m$ so that $I = (a) \oplus (b) = (a, b)$. Then $ab = 0$ and 
$$
I/I^2 \cong (a)/(a^2) \oplus (b)/(b^2)\cong R/[(a) + (0):_R a] \oplus R/[(b) + (0):_R b],
$$
while $I/I^2 \cong (R/I)^{\oplus 2}$, since $I \in \calX_R$ and $\mu_R(I) = 2$. Therefore, because $I = (a, b) \subseteq (a) + (0):_R a $ and $I \subseteq (b) + (0):_R b$, we get 
$I = (a) + (0):_R a = (b) + (0):_R b$.
In contrast, we have 
$$
I^2 = (a^2, b^2) = (a+b)I,
$$
hence $a+b$ is a non-zero divisor on $R$, since $\sqrt{I} = \m$. We also have the following. 

\begin{claim*}
$(0):_Ra^2 = (0):_Ra$ and $(0):_Rb^2 = (0):_Rb$.
\end{claim*}
\begin{proof}[Proof of Claim]
$(0):_Ra \subseteq (0):_Ra^2$ is clear. Let $x \in (0):_Ra^2$. Since $(a+b)ax=a^2x + abx = 0$ and $a+b$ is a non-zero divisor on $R$, we have $ax = 0$, which shows $(0):_Ra^2 = (0):_Ra$. Similarly, $(0):_Rb^2 = (0):_Rb$. 
\end{proof}
Let $x\in (0):_Ra$.
Because $x \in I = (a, b)$, we write $x = ax_1 + bx_2$ ($x_i \in R$). Then 
$$
0 = ax = a^2x_1 + abx_2 = a^2x_1,
$$
which shows that $x_1 \in (0):_Ra^2 = (0):_Ra$ by Claim. Consequently, we have $x = bx_2 \in (b)$, so that $(0):_Ra = (b)$. We also get $(0):_Rb = (a)$ as well.

(2) $\Rightarrow$ (1)     Because $ab = 0$, we have $I^2 = (a+b)I$. Hence $a+b$ is a non-zero divisor on $R$. Let $x \in (a) \cap (b)$. Then $(a+b)x = 0$, that is $x =0$. Therefore $I = (a) \oplus (b)$ and we have 
$$
 I/I^2 \cong (a)/(a^2) \oplus (b)/(b^2)\cong R/[(a) + (0):_R a] \oplus R/[(b) + (0):_R b] = R/I \oplus R/I,
$$
which shows that $I \in \calX_R$, as claimed.
\end{proof}

We now assume that $(S, \n)$ is a regular local ring with $\dim S = 2$, and let $0 \neq f \in \n$ and $R = S/(f)$. We then have the following, which determine all decomposable Ulrich ideals in a one-dimensional hypersurface ring.

\begin{thm}\label{4.2}
Assume that $f = p_1^{e_1}p_2^{e_2}\cdots p_{\ell}^{e_{\ell}}$ $($$\ell \ge 1, e_i\ge 1$$)$ where $p_1, p_2, \ldots, p_{\ell}$ are distinct prime elements of $S$.
Set $\Lambda = \{1, 2, \cdots, \ell \}$. For $\emptyset \neq J \subsetneq \Lambda$, we define $\alpha _J = \prod_{j\in J}p_j^{e_j}$ and $\beta_J = \prod_{j \in \Lambda \setminus J}p_j^{e_j}$. Then 
$$\{ I \in \calX_R \mid \text{$I$ is decomposable}\  \} =
\{ (\overline{\alpha_J}, \overline{\beta_J}) \mid \emptyset \neq J \subsetneq \Lambda \}.
$$
\end{thm}
\begin{proof}
Suppose that $\emptyset \neq J \subsetneq \Lambda$, and set $a  = \alpha_{J} + \beta_{J}, $ $b = \beta_J$. Then $a, b$ is a system of parameters of $S$, since $\alpha_J, \beta_J$ is also a system of parameters of it, and we have
$$a^2\cdot0 + ab\cdot (-1) + b^2 = -\alpha_J\beta_J = -f.$$
Thus $(\overline{a}, \overline{b}) = (\overline{\alpha_J}, \overline{\beta_J}) \in \calX_R$ by Proposition \ref{2.1}, and  $(\overline{\alpha_J}, \overline{\beta_J}) = (\overline{\alpha_J})\oplus (\overline{\beta_J})$. 

Conversely, suppose that $I \in \calX_R$ and $I$ is decomposable. Then, because $R$ is a Gorenstein ring, $\mu_R(I) = 2$ by Theorem \ref{1.2}. We can choose $a, b \in \n$ so that $I = (\overline{a}, \overline{b})$, $(0):_R\overline{a} = (\overline{b})$, and $(0):_R\overline{b} = (\overline{a})$ by Proposition \ref{4.1}. Since $\overline{a}\overline{b}= 0$ in $R$, we write $ab = \rho f$ with $\rho \in S$. We note that $a, b$ are relatively prime because $a, b$ is a system of parameters of $S$ by Lemma \ref{2.3}. Therefore, it suffices to show that $\rho \in U(S)$. Assume that $\rho \in \n$. Then $\rho = p\rho'$ for some prime element $p$ of $S$ and $\rho' \in S$, hence $ab = p\rho'f \in (p)$, and we may assume that $a \in (p)$. Thus, writing $a = pa'$ with $a' \in S$, we get $a'b = \rho'f$, which means $\overline{a'} \in (0):_R\overline{b} = (\overline{a})$. This is impossible since $p \notin U(S)$.
\end{proof}

The following is a direct consequence of Theorem \ref{4.2}.

\begin{cor}\label{4.3}
Suppose that $R = k[[X, Y]]/(X^kY)$, where $k>0$ and $k[[X, Y]]$ is a formal power series ring over a field $k$. Then 
$$
\{ I \in \calX_R \mid I\ \text{is decomposable}\  \} = \{ (x^k, y) \}
$$
where $x, y$ denote the images of $X, Y$ in $R$.
\end{cor}


\section{The case where $R= k[[X, Y]]/(f)$}

In this section, let $S = k[[X, Y]] $ be a formal power series ring over a field $k$, and $R = S/(f)$ with $f\in \n = (X, Y)$. By using Theorem \ref{2.2} and Corollary \ref{4.3}, we explore the set $\calX_R$, when $f = Y^k$ or $X^{k-1}Y$$($$k\ge2$$)$.
Let $x, y$ denote the images of $X, Y$ in $R$.

Firstly, we assume that $f = Y^k$ and $k\ge2$. Let $I \in \calX_R$. Remember that $\mu_R(I) = 2$, since $R$ is a Gorenstein ring. 

\begin{prop}\label{4.4}
$I =(\overline{a}, \overline{b})$ and $I^2 = \overline{a}I$ for some $a = X^n + a_1Y$ and $b = b_1Y$, where $n>0$, $a_1, b_1\in S$. Therefore $Y^{k-1} \in (a, b)$.
\end{prop}

\begin{proof}
Let us write $I = (\alpha, \beta)$ with $I^2 = \alpha I$ $(\alpha, \beta \in R)$. We set $$A = I:I = \{ \varphi\in Q(R) \mid \varphi I\subseteq I \} \subseteq Q(R),$$
where $Q(R)$ denotes the total ring of fractions of $R$. Then $A = \frac{I}{\alpha} = R + R\frac{\beta}{\alpha}$, since $I^2 = \alpha I$. In contrast, let $D = k[[x]] \subseteq R$ and $K = Q(D)$. Then, since $A$ is a module finite birational extension of $R$ and $Q(R) = K[Y]/(Y^k)$, we have 
$$
R\subseteq A = R + R\frac{\beta}{\alpha} \subseteq \overline{R} = D + \sum_{i = 1}^{k-1} Ky^i, 
$$
where $\overline{R}$ denotes the integral closure of $R$ in $Q(R)$. Because $\frac{\beta}{\alpha} \in D + \sum_{i =1}^{k-1} Ky^i$, we write $\frac{\beta}{\alpha} = d + \rho$ with $d\in D$ and $\rho \in \sum_{i =1}^{k-1} Ky^i$. Therefore, since $\frac{\beta-\alpha d}{\alpha} = \frac{\beta}{\alpha} - d = \rho$ and $A = R + R\rho$, replacing $\beta$ with $\beta-\alpha d$, from the beginning we may assume that $\frac{\beta}{\alpha} \in \sum_{i =1}^{k-1} Ky^i$. Hence $y^{k-1}\beta = 0$, since $y^{k-1}\cdot \frac{\beta}{\alpha} = 0$ in $R$. Therefore, we have $y^{k-1} \in I$, because $(\alpha):_R \beta = I$ (remember that $I/(\alpha) \cong R/I$). Let $a, b\in S$ such that $\overline{a} = \alpha$, $\overline{b} = \beta$ in $R$. 
Then $a, b$ is a system of parameters of $S$ by Lemma \ref{2.3}.
Since $bY^{k-1} \in (Y^k)$ in $S$, we get $b \in (Y)$, and that $a \notin (Y)$. Consequently, we have that $a = \varepsilon X^n + a_1Y$ and $b = b_1Y$ with $n>0$, $a_1, b_1\in S$, and $\varepsilon \in U(S)$, and may assume $\varepsilon = 1$. We also have $Y^{k-1} \in (a, b)$, since $Y^{k-1} \in (a, b) + (Y^k)$.
\end{proof}

\begin{prop}\label{5.1}$($\cite[Example 4.8]{GIK}$)$
Suppose that $R = k[[X, Y]]/(Y^2)$. Then 
$$
\calX_R = \{ (x^{\ell}, y) \mid \ell>0\}.
$$
\end{prop}

\begin{proof}
Thanks to Corollary \ref{2.4}, $(x^{\ell}, y) \in \calX_R$ for any $\ell>0$. Conversely, suppose that $I \in \calX_R$. Then $I =(\overline{a}, \overline{b})$ for some $a = X^n + a_1Y$ and $b = b_1Y$ with $n>0$, $a_1, b_1\in S$, and $Y \in (a, b)$ by Proposition \ref{4.4}. Therefore, $(a, b) = (a, b, Y) = (X^n, Y)$, as desired.
\end{proof}

If $k$ is odd, we have the following family of Ulrich ideals.

\begin{prop}\label{4.5}
Suppose that $k = 2m +1$ $(m\ge1)$. Let $\ell>0$ and $\varepsilon \in U(S)$. We consider the ideal $I = (x^{2\ell} + \overline{\varepsilon}y, x^{\ell}y^m)$ of $R$. Then the following assertions hold true.
\begin{enumerate}[$(1)$]
\item
$I \in \calX_R$.
\item
Let $\ell'>0$, $\varepsilon'\in U(S)$ and suppose that $I = (x^{2\ell'} + \overline{\varepsilon'}y, x^{\ell'}y^m)$. Then $\ell = \ell'$ and $\varepsilon \equiv \varepsilon'$ $\mod$ $\n$.
\end{enumerate}
\end{prop}
\begin{proof}
(1) Let $a = X^{2\ell} + \varepsilon Y$ and $b = X^{\ell}Y^{m}$. Then $a, b$ is a system of parameters of $S$, and setting $\varphi = -\varepsilon^{-1}Y^{2m -1}$, $\psi = \varepsilon X^{\ell}Y^{m-1}$, and $\delta = - 1$, we have 
$$
a^2\varphi +ab\psi +b^2 = \delta Y^{2m +1}, 
$$
so that $I = (\overline{a}, \overline{b}) \in \calX_R$ by Proposition \ref{2.1}.

(2) Let $\ell, \ell' >0$, $\varepsilon, \varepsilon' \in U(S)$, and assume that 
$$
(x^{2\ell} + \overline{\varepsilon}y, x^{\ell}y^m) = (x^{2\ell'} + \overline{\varepsilon'}y, x^{\ell'}y^m).
$$ 
Then $(X^{2\ell} + \varepsilon Y, X^{\ell}Y^m) = (X^{2\ell'} + \varepsilon'Y, X^{\ell'}Y^m)$ by Lemma \ref{2.3}, hence we have $\ell = \ell'$ by comparing the colength of the ideals. We write $X^{2\ell} + \varepsilon Y = (X^{2\ell} + \varepsilon'Y)\xi + (X^{\ell}Y^m)\eta$ with $\xi, \eta\in S$. Then $X^{2\ell}(1 -\xi) = Y(-\varepsilon + \varepsilon'\xi + X^{\ell}Y^{m-1}\eta)$, whence 
$$
1- \xi = Y\rho \ \text{and}\ -\varepsilon + \varepsilon'\xi + X^{\ell}Y^{m-1}\eta = X^{2\ell}\rho
$$
for some $\rho \in S$. Therefore, $1\equiv \xi$ and $-\varepsilon + \varepsilon'\xi \equiv 0$ mod $\n$, that is $\varepsilon \equiv \varepsilon'$.
\end{proof}

As a consequence, we get the following. 

\begin{thm}\label{4.6}
Suppose that $R = k[[X, Y]]/(Y^3)$. Then
$$
\calX_R = \{ (x^{2\ell} + \overline{\varepsilon}y, x^{\ell}y) \mid \ell> 0, \varepsilon \in U(S)\}.
$$
\end{thm}
\begin{proof}
The inclusion ($\supseteq$) follows from Proposition \ref{4.5}. Suppose that $I \in \calX_R$. By Proposition \ref{4.4}, $I =(\overline{a}, \overline{b})$ for some $a = X^n + a_1Y$ and $b = b_1Y$ with $n>0$, $a_1, b_1\in S$. We notice that $\ell_R(R/(\overline{a})) = 2\cdot\ell_R(R/I)$, since $I/(\overline{a}) \cong R/I$, and $\ell_R(R/(\overline{a})) = \ell_S(S/(a, Y^3)) = 3n$.  If $b_1\notin \n$, then $(a, b) = (X^n, Y)$, whence 
$
\ell_R(R/I) = \ell_S(S/(a, b)) = n.
$
This implies that $3n =2n$, which is impossible. Hence $b_1 \in \n$. If $b_1 \in (Y)$, then $y\overline{b} = 0$ in $R$, thus $y \in (\overline{a}):_R\overline{b} = I$. This implies that $Y\in (a, b)$ and $(a, b) = (X^n, Y)$, which is also impossible. Therefore, since $b_1\in \n\setminus (Y)$, we write $b_1 = \tau X^{\ell} +b_2Y$ with $\ell>0$, $\tau\in U(S)$, and $b_2 \in S$. Because $Y^2 \in (a, b)$ by Proposition \ref{4.4}, we have $(a, b) = (a, b, Y^2) = (X^n +a_1Y, X^{\ell}Y, Y^2)$, whence $(a, b) = (X^n +a_1Y, X^{\ell}Y)$ or $(X^n +a_1Y, Y^2)$, since $(a, b) \nsubseteq (Y)$. We then have $(a, b) = (X^n +a_1Y, X^{\ell}Y)$. Indeed, if $(a, b) = (X^n +a_1Y, Y^2)$, then $2\cdot \ell_R(R/I) = 2\cdot \ell_S(S/(X^n +a_1Y, Y^2)) =4n \neq 3n,$ which is impossible. Therefore, we may assume that $b_1 = X^{\ell}$. In addition, we have the following.

\begin{claim*}
$a_1 \in U(S)$.
\end{claim*}
\begin{proof}[Proof of Claim]
Because $(\overline{a}, \overline{b}) \in \calX_R$,  
$$
a^2\varphi + ab\psi +b^2 = \varepsilon Y^3
$$
for some $\varphi, \psi \in S$ and $\varepsilon \in U(S)$ by Theorem \ref{2.2}. Since $a^2\varphi \in (Y)$ and $a\notin (Y)$, $\varphi = Y \varphi_1$ for some $\varphi_1 \in S$. Expanding the equation, we have
$$
a_1^2\varphi_1Y^2 + X^{2\ell}Y +2a_1\varphi_1 X^{2\ell}Y +a_1\psi X^{\ell}Y + \varphi_1 X^{4\ell} + \psi X^{3\ell} = \varepsilon Y^2.
$$
Therefore, $Y^2(a_1^2\varphi_1 - \varepsilon) \in (X)$, so that $a_1^2\varphi_1 - \varepsilon \in (X)$, whence $a_1\in U(S)$.
\end{proof}

It suffices to show that $n = 2\ell$. In fact, we have 
$$
\ell_R(R/I) = \ell_S(S/(X^n +a_1Y, X^{\ell}Y)) = \ell + n,
$$
while $\ell_R(R/(\overline{a})) = 3n$. Consequently, $3n = 2(\ell +n)$, whence $n =2\ell$. This completes the proof of Theorem \ref{4.6}.
\end{proof}

Similarly, if $k$ is even, we have the following. 

\begin{prop}\label{4.7}
Suppose that $k = 2m$ $(m\ge2)$. Then the following assertions hold true.
\begin{enumerate}[$(1)$]
\item
$\{ I \in\calX_R \mid y^m \in I \} = \{ (x^{\ell} + \alpha y, y^m) \mid \ell>0, \alpha\in R \}$. 
\item
Let $\ell, \ell' >0$, $\alpha, \alpha' \in R$ and suppose that $(x^{\ell} + \alpha y, y^m) = (x^{\ell'} + \alpha' y, y^m)$. Then $\ell = \ell'$ and $\alpha \equiv \alpha'$ $\mod$ $\m = \n/(Y^{2m})$. 
\end{enumerate}
\end{prop}
\begin{proof}
(1)  The inclusion ($\supseteq$) follows from Corollary \ref{2.4}. Suppose that $I \in \calX_R$. $I =(\overline{a}, \overline{b})$ for some $a = X^n + a_1Y$ and $b = b_1Y$ with $n>0$, $a_1, b_1\in S$. Since $\ell_R(R/(\overline{a})) = 2\cdot \ell_R(R/I)$ and $\ell_R(R/(\overline{a})) = \ell_S(S/(X^n + a_1Y, Y^{2m})) = 2mn$, we have $\ell_R(R/I) = mn$. In contrast, because $y^m \in I$, $$mn = \ell_R(R/I) = \ell_S(S/(a, b)) = \ell_S(S/(a, b, Y^m)) \le \ell_S(S/(X^n + a_1Y, Y^m)) = mn,$$
hence $(a, b) = (X^n + a_1Y, Y^m)$, as desired. The Assertion (2) follows from the same technique as in the proof of Proposition \ref{4.5} (2).
\end{proof}

\begin{cor}\label{4.8}
Suppose that $R = k[[X, Y]]/(Y^4)$. Then
$$
\{ I \in\calX_R \mid y^2 \in I \} = \{ (x^{\ell} + \alpha y, y^2) \mid \ell>0, \alpha\in R \}.
$$
\end{cor}

For a moment, suppose that $k = 4$. Let $I \in\calX_R$ and assume that $y^2 \notin I$. Then $I =(\overline{a}, \overline{b})$ and $I^2 = \overline{a}I$ for some $a = X^n + a_1Y$ and $b = b_1Y$, where $n>0$, $a_1, b_1\in S$ by Proposition \ref{4.4}. With this notation, we get the following.

\begin{lem}\label{4.9}
$b_1 = X^p + b_2Y$ with $0<p<n$ and $b_2 \in S$.
\end{lem}

\begin{proof}
Because $y\notin I$, $b_1 \notin U(S)$. We then have $b_1\in \n \setminus (Y)$. Indeed, if $b_1\in (Y)$, then $y^2\overline{b} = 0$ in $R$, whence $y^2\in I$. This is impossible. Therefore $b_1 = \tau X^p + b_2 Y$ with $p>0$, $b_2 \in S$, and $\tau \in U(S)$, and may assume $\tau = 1$. Assume $p\ge n$. Then, because 
$$b = X^pY + b_2Y^2 \underset{\mod a}{\equiv} X^{p-n}Y(-a_1Y) + b_2Y^2 \in (Y^2),$$
we have $y^2\in (\overline{a}):_R \overline{b} = I$, which is impossible. Therefore $0<p<n$.
\end{proof}

\begin{thm}\label{4.10}
Suppose that $R =k[[X, Y]]/(Y^4)$. Let $I\in\calX_R$ and assume that $y^2\notin I$. We set $I = (\overline{a}, \overline{b})$ with $a, b \in S$. Then the following assertions hold true. 
\begin{enumerate}[$(1)$]
\item
$(a, b) = (X^n + a_1Y, Y(X^p + b_2Y))$ with $0<p<n$, $a_1\in\n$, and $b_2\in U(S)$.
\item
If $a_1 \in (Y)$, then $\operatorname{ch}k =2$.
\item
If ${\rm ch}$$k \neq 2$, then $(a, b) = (X^n + \alpha X^rY, Y(X^p + b_2Y))$ with $0<r<p<n$, $n-p\le r$, and $\alpha, b_2 \in U(S)$.
\end{enumerate} 
\end{thm}
\begin{proof}
(1)  Thanks to Lemma \ref{4.9}, $(a, b) = (X^n + a_1Y, Y(X^p + b_2Y))$ with $0<p<n$ and $a_1, b_2 \in S$. We may assume $a = X^n + a_1Y$ and $b = Y(X^p + b_2Y)$. Because $\ell_R(R/(\overline{a})) = 2\cdot \ell_R(R/I)$ and $\ell_R(R/(\overline{a})) = \ell_S(S/(X^n + a_1Y, Y^4)) = 4n$, we have $$2n = \ell_R(R/I) = \ell_S(S/(X^n + a_1Y, Y(X^p + b_2Y))) = n + \ell_S(S/(X^n + a_1Y, X^p + b_2Y)),$$
so that $\ell_S(S/(X^n + a_1Y, X^p + b_2Y)) = n$. If $a_1 \in U(S)$, then $(X^n + a_1Y, X^p + b_2Y) = (X^n + a_1Y, X^p(1-a_1^{-1}b_2X^{n-p})) = (X^p, Y)$, hence $n = \ell_S(S/(X^n + a_1Y, X^p + b_2Y)) = p$, which is impossible. Therefore $a_1 \in \n$. In contrast, we have 
$$
a^2\varphi + ab\psi +b^2 = \varepsilon Y^4
$$
for some $\varphi, \psi \in S$ and $\varepsilon \in U(S)$ by Theorem \ref{2.2}. Then $\varphi = Y \varphi_1$ for some $\varphi_1 \in S$, since $a^2\varphi \in (Y)$ and $a\notin (Y)$. From the equation, we get
\begin{equation*}
\begin{split}
\varepsilon Y^3 = &\  b_2^2Y^3\\
& + a_1^2\varphi_1Y^2 + a_1b_2\psi Y^2 + 2b_2X^pY^2\\
& +  2a_1\varphi_1X^nY + b_2\psi X^nY + a_1\psi X^pY + X^{2p}Y\\
& + \varphi_1X^{2n} + \psi X^{n+p}.
\end{split}
\end{equation*}
Hence $X^{n + p}(\varphi_1 X^{n-p} + \psi) \in (Y)$, so that $\varphi_1 X^{n-p} + \psi \in (Y)$, whence $\psi \in \n$.
Similarly, $Y^2(-\varepsilon Y + b_2^2Y + a_1^2\varphi_1 + a_1b_2\psi) \in (X)$, so that $-\varepsilon Y + b_2^2Y + a_1^2\varphi_1 + a_1b_2\psi \equiv 0$ mod $(X)$.
Because $a_1, \psi \in \n$, $-\varepsilon Y + b_2^2Y  \equiv 0$ mod $(X, Y^2)$, whence $b_2 \in U(S)$.

(2)  Assume $a_1 \in (Y)$. Then, because $0<p<n$ and $\psi \in \n$, we have $2b_2X^pY^2\in (X^{p+1}, Y^3)$, therefore $\operatorname{ch}k =2$, since $b_2 \in U(S)$.

(3) Suppose that $\operatorname{ch}k \neq 2$. Then $a_1 \in \n \setminus (Y)$ by Assertions (1), (2). We write $a_1 = \alpha X^r + a_2Y$ with $r>0$, $\alpha \in U(S)$, and $a_2 \in S$. If $r\ge p$, since $a = X^n + \alpha X^rY \underset{\mod b}{\equiv} X^n + (-\alpha b_2^{-1}X^{r-p}Y)Y$, then by replacing $\alpha X^r$ with $-\alpha b_2^{-1}X^{r-p}Y$, we would have to assume that $a_1 \in (Y)$, which is unreasonable. Hence $0<r<p<n$. Because
$$
a = X^n + \alpha X^rY + a_2Y^2 \underset{\mod b}{\equiv} X^n + \alpha X^rY - a_2b_2^{-1}X^pY = X^n + (\alpha - a_2b_2^{-1}X^{p-r})X^rY
$$ 
and $\alpha - a_2b_2^{-1}X^{p-r} \in U(S)$, we may assume that $a_2 = 0$. Since  $\ell_S(S/(X^n + a_1Y, X^p + b_2Y)) = n$ (see the proof of Assertion (1)), if $n>r+p$,
\begin{equation*}
\begin{split}
n & = \ell_S(S/(X^n + a_1Y, X^p + b_2Y)) = \ell_S(S/(X^n + \alpha X^rY, X^p + b_2Y))\\
& = \ell_S(S/(X^n - \alpha b_2^{-1}X^{r+p}, X^p + b_2Y)) = \ell_S(S/(X^{r+p}, X^p + b_2Y)) = r+p,
\end{split}
\end{equation*}
which makes a contradiction. Therefore, $n \le r+p$.
\end{proof}

Now we explore a concrete example.

\begin{ex}\label{4.11}
Suppose that $R =k[[X, Y]]/(Y^4)$. Let $p, n$ be integers such that $0<p<n$ and $2n\le 3p$. We set $a = X^n + 2X^{n-p}Y$, $b = Y(X^p +Y)$. Then the following assertions hold true. 
\begin{enumerate}[$(1)$]
\item
$I = (\overline{a}, \overline{b}) \in \calX_R$, for any characteristic of $k$.
\item
$y^2 \notin I$.
\end{enumerate} 
\end{ex}
\begin{proof}
(1)  We set $\varphi = -X^{3p-2n}Y$, $\psi = X^{2p-n}$, and $\varepsilon = 1$. Then $a, b$ is a system of parameters of $S$, and we have $a^2\varphi + ab\psi + b^2 = \varepsilon Y^4$; therefore, $I \in \calX_R$ by Proposition \ref{2.1}.

(2)  If $y^2 \in I$, then $Y^2 \in (a, b)$. We write $Y^2 = (X^n + 2X^{n-p}Y)\xi + Y(X^p + Y)\eta$ with $\xi, \eta \in S$. Hence, since $\xi = Y\xi_1$ for some $\xi_1 \in S$, we have $Y(1-2X^{n-p}\xi_1-\eta) = X^p(X^{n-p}\xi_1 +\eta)$, so that $1-2X^{n-p}\xi_1-\eta = \rho X^p$ and $X^{n-p}\xi_1 +\eta = \rho Y$ for some $\rho \in S$. This implies that $1 \equiv \eta$ and $\eta \equiv 0$ mod $\n$, which is impossible. 
\end{proof}

In what follows, we assume that $f = X^kY$ ($k\ge1$). Thanks to Corollary \ref{4.3}, $(x^k, y)$ is the only decomposable Ulrich ideal in $R$. Let $I \in \calX_R$ and $I$ is indecomposable. We begin with the following. 

\begin{prop}\label{4.12}
$I =(\overline{a}, \overline{b})$ and $I^2 = \overline{a}I$ for some $a = X^n + a_1Y$ and $b = b_1XY$, where $n>0$, $a_1, b_1\in S$ such that $a_1 \notin (X)$. In addition, $n<k$ if $k\ge2$.
\end{prop}
\begin{proof}
We identify $R \subseteq S/(X^k) \times S/(Y)$ and let $x_1, y_1$ (resp. $x_2$) denote the images of $X, Y$ (resp. $X$) in $S/(X^k)$ (resp. $S/(Y)$). Hence $S/(Y) = k[[x_2]]$ and $Q(R) = (K_1 + \sum_{i = 1}^{k-1}K_1x_1^i) \times K_2$, where $K_1 = Q(k[[y_1]])$ and $K_2 = Q(k[[x_2]])$. We set $A = I: I$. Then
$$
R \subseteq A \subseteq \overline{R} = (k[[y_1]] + \sum_{i = 1}^{k-1}K_1x_1^i) \times k[[x_2]],
$$
since $A$ is a module finite birational extension of $R$. Let us write $I = (\alpha, \beta)$ with $I^2 = \alpha I$. Then $A = R + R\frac{\beta}{\alpha}$. Remember now that $A$ is a local ring, since $A \cong I$ is indecomposable. Let $J, \m,$ and $J(\overline{R})$ denote the maximal ideals of $A, R$, and the Jacobson radical of $\overline{R}$. Then, since 
$$
k = R/\m \subseteq A/J \subseteq \overline{R}/J(\overline{R}) = k \times k,
$$
we have $R/\m = A/J$. Take $r \in R$ so that $\frac{\beta}{\alpha} \equiv r$ mod $J$. Then, replacing $\beta$ with $\beta - r\alpha$, we can assume that $\frac{\beta}{\alpha} \in J$. Since $J \subseteq J(\overline{R}) =  (y_1k[[y_1]] + \sum_{i = 1}^{k-1}K_1x_1^i) \times x_2k[[x_2]]$, we get $\frac{\beta}{\alpha} = r' + \rho$ for some $r' \in R$ and $\rho \in (\sum_{i = 1}^{k-1}K_1x_1^i) \times (0)$. Therefore, replacing $\beta$ with $\beta- \alpha r'$, from the beginning we may assume that $\frac{\beta}{\alpha} \in (\sum_{i = 1}^{k-1}K_1x_1^i) \times (0)$. Let us now write $\alpha = \overline{a}$ and $\beta = \overline{b}$ with $a, b \in S$. Then, since $\beta^k = 0$ in $R$, we have $b^k \in (X^kY)$, so that $b \in (XY)$. We write $b = b_1XY$ with $b_1 \in S$. Notice that $a, b$ is a system of parameters of $S$ by Lemma \ref{2.3}. Consequently, $a \notin (X)\cup (Y)$, so that we may assume that $a = X^n + a_1Y$ with $n>0$ and $a_1 \in S$ such that $a_1 \notin (X)$. If $k\ge2$, we have $X^{k-1} \in (a, b)$, since $x^{k-1} \in (\alpha):_R\beta = I$. Thus, because $X^{k-1} \in (a, b, Y) = (X^n, Y)$, we get $n<k$. 
\end{proof}

\begin{thm}\label{4.13}
Suppose that $R = k[[X, Y]]/(X^kY)$ with $1\le k \le 2$. Then
$$
\calX_R = \{ (x^k, y) \}.
$$
\end{thm}
\begin{proof}
Suppose that $I\in \calX_R$ and $I$ is indecomposable. Assume that $k = 1$. Then, since $\overline{R} = S/(X) \times S/(Y)$ and $\ell_R(\overline{R}/R) =1$, $A = \overline{R}$ where $A = I:I$, which is impossible because $A$ is a local ring (see the proof of Proposition \ref{4.12}). Assume that $k = 2$. By Proposition \ref{4.13}, $I =(\overline{a}, \overline{b})$ for some $a = X + a_1Y$ and $b = b_1XY$ with $a_1, b_1\in S$ such that $a_1 \notin (X)$. Since $X \in (a, b)$ (see the proof of Proposition \ref{4.12}), we can write $X = (X + a_1Y) \varphi + b_1XY \psi$ with $\varphi, \psi \in S$. Then $a_1Y\varphi \in (X)$ and $a_1 \notin (X)$, whence $\varphi \in (X)$. Therefore, writing $\varphi = X\varphi_1$ with $\varphi_1 \in S$, we get $1 = (X + a_1Y)\varphi_1 + b_1Y\psi \in \n$, which is impossible. Consequently, if $k\le2$, $R$ has no indecomposable Ulrich ideal. Thanks to Corollary \ref{4.3}, this completes the proof of this Theorem. 
\end{proof}

In what follows, suppose that $k \ge 3$. Let $I \in\calX_R$ and assume that $I$ is indecomposable. Then $I =(\overline{a}, \overline{b})$ and $I^2 = \overline{a}I$ for some $a = X^n + a_1Y$ and $b = b_1XY$ with $n>0$ and $a_1, b_1\in S$ such that $a_1 \notin (X)$ by Proposition \ref{4.12}. With this notation, we have the following.

\begin{proposition}\label{4.14}
The following assertions hold true. 
\begin{enumerate}[$(1)$]
\item
$n\le k-2$.
\item
If $k \ge 4$ and $n = k-2$, then $xy \in I$.
\end{enumerate}
\end{proposition}
\begin{proof}
Because $(\overline{a}, \overline{b}) \in \calX_R$, 
$$
a^2\varphi + ab\psi + b^2 = \varepsilon X^kY
$$
for some $\varphi, \psi \in S$ and $\varepsilon \in U(S)$ by Theorem \ref{2.2}.
Since $a^2\varphi \in (XY)$ and $a \notin (X)\cup (Y)$, $\varphi = XY\varphi_1$ for some $\varphi_1 \in S$. We then have
\begin{equation*}
\begin{split}
\varepsilon X^{k-1} = &\  a_1^2\varphi_1Y^2\\
& + 2a_1\varphi_1 X^nY + a_1b_1\psi Y + b_1^2 XY\ \   \cdots \text{(A)}\\
& + \varphi_1 X^{2n} + b_1\psi X^n.
\end{split}
\end{equation*}
(1) Assume that $n>k-2$. Then $ n= k-1$ by Proposition \ref{4.12}. Hence $X^{k-1}(\varepsilon -b_1\psi - \varphi_1 X^{k-1}) \in (Y)$, so that $\varepsilon -b_1\psi  \in \n$, whence $b_1 \in U(S)$. Therefore, we may assume that $b_1 = 1$. Since $a_1 \notin (X)$, we write $a_1 = \tau Y^{\ell} + a_2 X$ with $\ell \ge0$, $\a_2 \in S$, and $\tau \in U(S)$. We then have $(a, b) = (X^{k-1} + \tau Y^{\ell + 1} + a_2XY, XY) = (X^{k-1} + \tau Y^{\ell + 1}, XY)$. Thus, from the beginning we may assume $a_1 = \tau Y^{\ell}$. From the above equation (A), we get $\tau \psi Y^{\ell + 1} + XY \equiv 0$ mod $(X^2, Y^2)$, hence $\ell = 0$. 
In contrast, because $\ell_R(R/(\overline{a})) = 2\cdot \ell_R(R/I)$, we have
\begin{equation*}
\begin{split}
&\ell_R(R/(\overline{a})) = \ell_S(S/(X^{k-1} + \tau Y, X^kY)) = k + k-1= 2k - 1,\ \text{and}\\
&\ell_R(R/I) = \ell_S(S/(X^{k-1} + \tau Y, XY)) = 1 + k-1 = k.
\end{split}
\end{equation*}
Hence $2k-1= 2k$, which is impossible. Therefore $n \le k-2$.

(2)  Suppose that $k\ge4$ and $n = k-2$. From the equation (A), we have $X^{k-2}(\varepsilon X - \varphi_1 X^{k-2} - b_1\psi) \in (Y)$, whence $b_1\psi \equiv \delta X$ mod $(Y)$, where $\delta = \varepsilon  - \varphi_1 X^{k-3} \in U(S)$. Assume that $b_1 \in \n$. Then $\psi \in U(S)$ and $b_1 = \rho X + b_2Y$ for some $\rho \in U(S)$ and $b_2 \in S$. We may assume that $\rho = 1$. We also get $a_1Y(a_1 \varphi_1 Y + b_1\psi) \in (X)$ from the equation (A). Since $a_1 \notin (X)$, we have $a_1\varphi_1 Y + b_1\psi \in (X)$, so that $a_1\varphi_1 Y + b_2 \psi Y = Y(a_1\varphi_1 + b_2\psi) \in (X)$. Whence $b_2 \in (a_1, X)$(notice that $\psi \in U(S)$). Writing $b_2 = a_1\xi + X\eta$ with $\xi, \eta\in S$, we get
$$
b = XY(X + a_1\xi Y + \eta XY) \underset{\mod a}{\equiv} X^2Y(1- \xi X^{k-3} + \eta XY), 
$$
hence we may assume that $b = X^2Y$ ($b_2 = 0$). Let $\ell = \ell_S(S/(a_1, X))$. Then
\begin{equation*}
\begin{split}
&\ell_R(R/(\overline{a})) = \ell_S(S/(X^{k-2} + a_1 Y, X^kY)) =  k (\ell + 1) + k - 2 = k\cdot \ell + 2k - 2,\ \text{and}\\
&\ell_R(R/I) = \ell_S(S/(X^{k-2} + a_1 Y, X^2Y)) = 2(\ell + 1) + k-2 = 2\ell + k.
\end{split}
\end{equation*}
Since $\ell_R(R/(\overline{a})) = 2\cdot \ell_R(R/I)$, we have $k\cdot \ell + 2k - 2 = 2(2\ell + k)$, so that $(k-4)\ell = 2$.
Thus, $k = 6$, $\ell = 1$ or $k =5$, $\ell =2$.

If $k = 6$ and $\ell = 1$, we can write $a_1 = \tau Y + a_2X$ with $\tau \in U(S)$ and $a_2 \in S$ (notice that $\ell = \ell_S(S/(a_1, X))$). From the equation (A), we get $\tau \psi XY^2 \equiv 0$ mod $(X^2, Y^3)$, which makes a contradiction.

If $k =5$ and $\ell =2$, we can write $a_1 = \tau Y^2 + a_2X$ with $\tau \in U(S)$ and $a_2 \in S$. Similarly, we get $\tau \psi XY^3 \equiv 0$ mod $(X^2, Y^4)$, which is impossible. Consequently, we have $b_1 \in U(S)$, therefore $xy \in I$.
\end{proof}

We get the following family of Ulrich ideals.

\begin{prop}\label{4.15}
Suppose that $k\ge 3$. Then the following assertions hold true.
\begin{enumerate}[$(1)$]
\item
$
\{ I \in \calX_R \mid xy \in I \} = \{ (x^{k-2} + \overline{\varepsilon}y, xy) \mid \varepsilon \in U(S) \}.
$
\item
Let $\varepsilon, \varepsilon' \in U(S)$ and suppose that $(x^{k-2} + \overline{\varepsilon}y, xy) = (x^{k-2} + \overline{\varepsilon'}y, xy)$. Then $\varepsilon \equiv \varepsilon'$ mod $\n$.
\end{enumerate}
\end{prop}

\begin{proof}
(1) Let $a = X^{k-2} + \varepsilon Y$ with $\varepsilon \in U(S)$ and $b = XY$. Then $a, b$ is a system of parameters of $S$. Setting $\varphi = 0$, $\psi = - \varepsilon^{-1}X$, and $\delta = - \varepsilon^{-1}$, we have $a^2\varphi + ab\psi + b^2 = \delta X^kY$, thus $(\overline{a}, \overline{b}) \in \calX_R$ by Proposition \ref{2.1}. Conversely, suppose that $I \in \calX_R$ and $xy \in I$. Then $I =(\overline{a}, \overline{b})$ and $I^2 = \overline{a}I$ for some $a = X^n + a_1Y$ and $b = b_1XY$ with $n>0$ and $a_1, b_1\in S$ by Proposition \ref{4.12}, and $XY\in (a, b)$, hence $(a, b) = (a, XY)$. Let  $\ell = \ell_S(S/(a_1, X))$. Because $\ell_R(R/(\overline{a})) = 2\cdot \ell_R(R/I)$,
\begin{equation*}
\begin{split}
&\ell_R(R/(\overline{a})) = \ell_S(S/(X^n + a_1Y, X^kY)) = k\cdot(\ell + 1) + n,\ \text{and}\\ 
&\ell_R(R/I) = \ell_S(S/(X^n + a_1Y, XY)) = \ell + 1 + n,
\end{split}
\end{equation*}
we have $k\cdot(\ell + 1) + n = 2(\ell + 1 + n)$, so that $(k-2)\ell = n - (k-2)$. Since $k \ge 3$ and $k-2 \ge n$ (Proposition \ref{4.14}), we get $n = k-2$ and $\ell = 0$, therefore $(a, b) = (X^{k-2} + a_1Y, XY)$ with $a_1 \in U(S)$ as desired. The Assertion (2) follows from the same technique as in the proof of Proposition \ref{4.5} (2).
\end{proof}

Let $I \in\calX_R$ and assume that $I$ is indecomposable. We choose $a = X^n + a_1Y$ and $b = b_1XY$ as in Proposition \ref{4.12}. We then have the following.

\begin{prop}\label{4.16}
The following assertions hold true. 
\begin{enumerate}[$(1)$]
\item
If $n = 1$, then $k$ is odd, and $(a, b) = (X + \varepsilon Y^{\ell}, XY^p)$ where $\varepsilon \in U(S)$ and $\ell, p>0$ such that $(k-2)\ell = 2p- 1$.
\item
Suppose that $k$ is odd. Let $\ell, p>0$ such that $(k-2)\ell = 2p- 1$ and $\varepsilon \in U(S)$. Then $(x + \overline{\varepsilon} y^{\ell}, xy^p) \in \calX_R$.
\item
Let $\ell, p>0$ $($resp. $\ell', p'>0$$)$ such that $(k-2)\ell = 2p- 1$ $($resp. $(k-2)\ell' = 2p'- 1$$)$ and $\varepsilon, \varepsilon' \in U(S)$. If $(x + \overline{\varepsilon} y^{\ell}, xy^p) = (x + \overline{\varepsilon'} y^{\ell'}, xy^{p'})$, then $\ell = \ell'$, $p=p'$, and $\varepsilon \equiv \varepsilon'$ mod $\n$.
\end{enumerate}

\end{prop}

\begin{proof}
(1) Suppose that $n=1$. Since $(a, Y) = \n$ and $S/(a)$ is a DVR, $b_1 = \rho Y^{p-1} + ab_2$ for some $p > 0$, $\rho \in U(S)$, and $b_2 \in S$ (notice that $b_1 \notin (a)$, since $b \notin (a)$). Then $(a, b) = (a, XY^p)$. In contrast, because $a_1 \notin (X)$, we can write $a_1 = \tau Y^{\ell-1} + a_2 X$ for some $\ell>0$ and $a_2 \in S$.
We then have  $a = X + a_2XY + \tau Y^{\ell} = (1 + a_2Y)X + \tau Y^{\ell }$, hence we may assume $a = X + \varepsilon Y^{\ell}$ with $\ell >0$ and $\varepsilon \in U(S)$. Now notice that $\ell_S(S/(a, X^kY)) = \ell_S(S/(X + \varepsilon Y^{\ell}, X^kY)) = k\ell + 1$ and $\ell_S(S/(a, b)) = \ell_S(S/(X + \varepsilon Y^{\ell}, XY^p)) = \ell + p$, so that $k\ell + 1 = 2(\ell + p)$, whence $(k-2)\ell = 2p - 1$ and $k$ is odd.

(2)  Let $a = X + \varepsilon Y^{\ell}$ and $b = XY^p$ with $\varepsilon \in U(S)$ and $\ell, p > 0$ such that $(k-2)\ell = 2p-1$. Then $a, b$ is a system of parameters of $S$. We set 
$$
\varphi =
\begin{cases}
-\varepsilon^{-1} XY\ \ \text{if $k = 3$}\\
\displaystyle\sum_{i = 0}^{k-4} (-1)^{i + k - 4}(i + 1)\varepsilon^{-(k-2)+ i}X^{k-2-i}Y^{i\ell + 1}\ \ \text{if $k\ge 5$}
\end{cases}, 
$$
$$
\psi =
\begin{cases}
Y^p\ \ \text{if $k =3$}\\
-(k-2)\varepsilon^{-1}XY^{p-\ell}\ \ \text{if $k\ge 5$}
\end{cases}, 
\text{and}\ \ 
\delta = 
(-1)^{k-4}\varepsilon^{-(k-2)}.
$$
Then we have $a^2\varphi + ab\psi + b^2 = \delta X^kY$, thus $(\overline{a}, \overline{b}) \in \calX_R$ by Proposition \ref{2.1}.
The Assertion (3) follows from the same technique as in the proof of Proposition \ref{4.5} (2).
\end{proof}

As a consequence, we get the following.

\begin{thm}\label{4.17}
The following assertions hold true. 
\begin{enumerate}[$(1)$]
\item
Suppose that $R = k[[X, Y]]/(X^3Y)$. Then
$$
\calX_R = \{ (x^3, y) \} \cup \{ (x + \overline{\varepsilon}y^{2p-1}, xy^p) \mid p>0, \varepsilon \in U(S) \}.
$$
\item
Suppose that $R = k[[X, Y]]/(X^4Y)$. Then
$$
\calX_R = \{ (x^4, y) \} \cup \{ (x^2 + \overline{\varepsilon}y, xy) \mid \varepsilon \in U(S) \}.
$$
\end{enumerate}
\end{thm}

\begin{proof}
These assertions readily follow from Corollary \ref{4.3}, Proposition \ref{4.14}, Proposition \ref{4.15}, and Proposition \ref{4.16}.
\end{proof}

\begin{acknowledgement} {\rm The author is grateful to Professor S. Goto for his helpful advice and useful comments. }
\end{acknowledgement}


\end{document}